\documentclass[a4paper]{amsart}

\usepackage[utf8]{inputenc}
\usepackage[english]{babel}

\usepackage[T1]{fontenc}
\usepackage{lmodern}  

\usepackage{my-math-symbol}
\usepackage{my-cleveref}
\usepackage{my-equation-numbering}
\usepackage{comment}
\usepackage{tikz}
\usetikzlibrary{cd}
\usepackage[colorinlistoftodos]{todonotes}

\usepackage[non-sorted-cites,initials,alphabetic,nobysame]{amsrefs}

\AtBeginDocument{%
	\def\MR#1{}
}

\title[Formulae of some global CR invariants]{Formulae of some global CR invariants for Sasakian $\eta$-Einstein manifolds}
\author{Yuya Takeuchi}
\address{Department of Mathematics \\ Graduate School of Science \\ Osaka University
	\\ 1-1 Machikaneyama-cho, Toyonaka, Osaka 560-0043, Japan}
\curraddr{Division of Mathematics \\ Institute of Pure and Applied Sciences \\ University of Tsukuba
	\\ 1-1-1 Tennodai, Tsukuba, Ibaraki 305-8571 Japan}
\email{ytakeuchi@math.tsukuba.ac.jp, yuya.takeuchi.math@gmail.com}



\thanks{This work was supported by JSPS Research Fellowship for Young Scientists
and JSPS KAKENHI Grant Number JP19J00063 and JP21K13792.}

\begin{document}

\begin{abstract}
	In this paper,
	we give explicit formulae of the Burns-Epstein invariant
	and global CR invariants via renormalized characteristic forms introduced by Marugame
	for Sasakian $\eta$-Einstein manifolds.
	As an application,
	we show that the latter invariants are algebraically independent.
\end{abstract}

\maketitle

\section{Introduction}
\label{section:introduction}

The biholomorphic equivalence problem
is one of the most major problems in several complex variables.
An approach to this problem is to study the CR structures on the boundaries of domains.
In this direction,
Fefferman~\cite{Fefferman1974} has shown that
two bounded strictly pseudoconvex domains in $\mathbb{C}^{n + 1}$ are biholomorphic
if and only if their boundaries are isomorphic as CR manifolds.
Since then,
it has been of great importance to construct and compute CR invariants.

Recently,
some researchers have introduced and studied invariants
for strictly pseudoconvex CR manifolds admitting a pseudo-Einstein contact form:
the total $Q^{\prime}$-curvature~\cites{Case-Yang2013,Hirachi2014},
the Burns-Epstein invariant~\cites{Burns-Epstein1990-Char,Marugame2016},
the total $\calI^{\prime}$-curvatures~\cites{Case-Gover2020,Marugame2021,Case-Takeuchi2023},
and global CR invariants via renormalized characteristic forms~\cite{Marugame2021}.
Note that the first and third ones are special cases of the last one.
In general,
it is hard to compute these invariants of a given CR manifold.
However,
we have already obtained explicit formulae of the first and third ones for Sasakian $\eta$-Einstein manifolds,
which are pseudo-Hermitian manifolds satisfying a strong Einstein condition.
This has been carried out for the total $Q^{\prime}$-curvature
by Case and Gover~\cite{Case-Gover2020} and the author~\cite{Takeuchi2018} independently,
and for the total $\calI^{\prime}$-curvatures by Marugame~\cite{Marugame2021} implicitly.
The purpose of this paper is to give formulae of the remaining cases,
the Burns-Epstein invariant and global CR invariants via renormalized characteristic forms,
for Sasakian $\eta$-Einstein manifolds.

We first consider global CR invariants via renormalized characteristic forms.
Let $\Omega$ be an $(n + 1)$-dimensional strictly pseudoconvex domain with boundary $M$.
Assume that $M$ admits a pseudo-Einstein contact form.
Take a Fefferman defining function $\rho$ of $\Omega$;
see~\cref{subsection:Fefferman-defining-function} for the definition.
The $(1, 1)$-form
\begin{equation}
	\omega_{+}
	= - d d^{c} \log (- \rho)
\end{equation}
defines a K\"{a}hler metric $g_{+}$ near the boundary,
where $d^{c} = (\sqrt{- 1} / 2)(\delb - \del)$.
The Chern connection with respect to $g_{+}$ diverges on the boundary since so does $g_{+}$.
However,
we obtain Burns-Epstein's renormalized connection,
which is smooth up to the boundary,
via a c-projective compactification~\cite{Cap-Gover2019}.
The corresponding curvature form is denoted by $\Theta$.
Let $\Phi$ be a (not necessarily homogeneous) $GL(n + 1, \mathbb{C})$-invariant polynomial of degree at most $n$.
Marugame~\cite{Marugame2021} has proved that
\begin{equation}
	\scrI_{\Phi}(M)
	= - \sum_{m = 0}^{n} \lp \int_{\rho < - \epsilon}
		\frac{d \log (- \rho) \wedge d^{c} \log (- \rho)}{2 \pi}
		\wedge  \pqty{ \frac{\omega_{+}}{2 \pi} }^{n - m}
		\wedge \Phi \pqty{ \frac{\sqrt{- 1}}{2 \pi} \Theta }
\end{equation}
is independent of the choice of $\rho$,
and gives a global CR invariant of $M$.
More precisely,
he has shown that $\scrI_{\Phi}$ is written as the integral of a linear combination of the complete contractions
of polynomials in the Tanaka-Webster torsion, curvature, and their covariant derivatives.
(Our definition is a little bit different from Marugame's one.
Our modification is to guarantee $\scrI_{\Phi_{1}} + \scrI_{\Phi_{2}} = \scrI_{\Phi_{1} + \Phi_{2}}$
for any $GL(n + 1, \mathbb{C})$-invariant polynomials $\Phi_{1}$ and $\Phi_{2}$ of degree at most $n$.)

Our first result is to give an explicit expression of $\scrI_{\Phi}$
in terms of $\Phi$, the Einstein constant, and the Chern tensor $\tensor{S}{_{\alpha}^{\beta}_{\rho}_{\ovxs}}$
for Sasakian $\eta$-Einstein manifolds.
A \emph{Sasakian $\eta$-Einstein manifold}
is a pseudo-Hermitian manifold $(S, T^{1, 0} S, \eta)$ of dimension $2 n + 1$
such that the Tanaka-Webster torsion $\tensor{A}{_{\alpha}_{\beta}}$
and Ricci curvature $\tensor{\Ric}{_{\alpha}_{\ovxb}}$ satisfy
\begin{equation}
	\tensor{A}{_{\alpha}_{\beta}}
	= 0,
	\qquad
	\tensor{\Ric}{_{\alpha}_{\ovxb}}
	= (n + 1) \lambda \tensor{l}{_{\alpha}_{\ovxb}},
\end{equation}
where $\lambda \in \mathbb{R}$ and $\tensor{l}{_{\alpha}_{\ovxb}}$ is the Levi form.
We call the constant $(n + 1) \lambda$ the Einstein constant of $(S, T^{1, 0} S, \eta)$.

\begin{theorem}
\label{thm:Marugame's-CR-invariant-for-Sasakian-eta-Einstein-manifolds}
	Let $(S, T^{1, 0} S, \eta)$ be a closed $(2 n + 1)$-dimensional Sasakian $\eta$-Einstein manifold
	with Einstein constant $(n + 1) \lambda$.
	For a $GL(n + 1, \mathbb{C})$-invariant polynomial $\Phi$ of degree at most $n$,
	define a $GL(n, \mathbb{C})$-invariant polynomial $\Phi^{\prime}$ by
	\begin{equation}
		\Phi^{\prime}(A)
		= \Phi
		\begin{pmatrix}
			A & 0 \\
			0 & 0
		\end{pmatrix}
		.
	\end{equation}
	Then
	\begin{equation}
		\scrI_{\Phi}(S)
		= \sum_{m = 0}^{n} \int_{S} \pqty{ - \frac{\lambda}{2 \pi} \eta }
			\wedge \pqty{ - \frac{\lambda}{2 \pi} d \eta }^{n - m}
			\wedge \Phi^{\prime} \pqty{ \frac{\sqrt{- 1}}{2 \pi} \Xi },
	\end{equation}
	where
	\begin{equation}
		\tensor{\Xi}{_{\alpha}^{\beta}}
		= \tensor{S}{_{\alpha}^{\beta}_{\rho}_{\ovxs}} \tensor{\theta}{^{\rho}} \wedge \tensor{\theta}{^{\ovxs}}.
	\end{equation}
\end{theorem}

A typical example of Sasakian $\eta$-Einstein manifolds
is a circle bundle over a K\"{a}hler-Einstein manifold.
Let $Y$ be an $n$-dimensional complex manifold
and $(L, h)$ a Hermitian holomorphic line bundle over $Y$
such that
\begin{equation}
	\omega = - \sqrt{-1} \Theta_{h} = d d^{c} \log h
\end{equation}
defines a K\"{a}hler-Einstein metric on $Y$ with Einstein constant $(n + 1) \lambda$.
Consider the circle bundle
\begin{equation}
	S = \Set{ v \in L | h(v, v) = 1}
\end{equation}
over $Y$,
which is a real hypersurface in the total space of $L$.
The triple
\begin{equation}
	(S, T^{1, 0} S = T^{1, 0} L |_{S} \cap (T S \otimes \mathbb{C}), \eta = d^{c} \log h |_{S})
\end{equation}
is a $(2 n + 1)$-dimensional Sasakian $\eta$-Einstein manifold with Einstein constant $(n + 1) \lambda$
and called the \emph{circle bundle associated with $(Y, L, h)$};
see~\cref{subsection:Sasakian-manifolds} for details.
In this case,
we can calculate $\scrI_{\Phi}$ in terms of the Einstein constant and characteristic numbers of $Y$.
In order to simplify computation,
we introduce some invariant polynomials.
Let $m$ be a positive integer.
Define a $GL(m, \mathbb{C})$-invariant homogeneous polynomial $\ch_{k}$ of degree $k$ by
\begin{equation}
	\ch_{k}(A)
	= \frac{1}{k !} \tr A^{k}.
\end{equation}
A \emph{partition} of $n$
is an $n$-tuple $\varsigma = (\varsigma_{1}, \dots, \varsigma_{n}) \in \mathbb{N}^{n}$ with $\sum_{k = 1}^{n} k \varsigma_{k} = n$.
The space of partitions of $n$ is written as $\Part(n)$.
For a partition $\varsigma$ of $n$,
set
\begin{equation}
\label{eq:generator-of-Marugame's-invariant}
	\Phi_{\varsigma}(A) = \prod_{k = 2}^{n} \ch_{k}(A)^{\varsigma_{k}},
\end{equation}
which is a $GL(m, \mathbb{C})$-invariant homogeneous polynomial of degree $n - \varsigma_{1}$.
We will write $\scrI_{\Phi_{\varsigma}}$ as $\scrI_{\varsigma}$ for simplicity.
It can be seen that any $\scrI_{\Phi}$ is written as a linear combination of $(\scrI_{\varsigma})_{\varsigma \in \Part(n)}$
(\cref{lem:linear-combination-of-Marugame's-invariant}).
Hence it suffices to compute $\scrI_{\varsigma}$.

\begin{theorem}
\label{thm:Marugame's-CR-invariants-of-Chern-character-type}
	Let $(L, h)$ be a Hermitian holomorphic line bundle over a closed $n$-dimensional complex manifold $Y$
	such that $\omega = - \sqrt{- 1} \Theta_{h}$ defines a K\"{a}hler-Einstein metric on $Y$
	with Einstein constant $(n + 1) \lambda$.
	Denote by $(S, T^{1, 0} S, \eta)$
	the circle bundle associated with $(Y, L, h)$.
	For a partition $\varsigma$ of $n$,
	\begin{equation}
		\scrI_{\varsigma}(S)
		= - \lambda \int_{Y} (\lambda c_{1}(L))^{\varsigma_{1}}
			\prod_{k = 2}^{n} \bqty{\sum_{j = 0}^{k} \frac{1}{(k - j) !} (\lambda c_{1}(L))^{k - j}
				\ch_{j}(T^{1, 0} Y \oplus \mathbb{C})}^{\varsigma_{k}}.
	\end{equation}
\end{theorem}

As an application of the above theorem,
we will show that $(\scrI_{\varsigma})_{\varsigma \in \Part(n)}$ are essentially different invariants.

\begin{theorem}
\label{thm:algebraically-independence}
	The invariants $(\scrI_{\varsigma})_{\varsigma \in \Part(n)}$ are algebraically independent over $\mathbb{C}$
	--- that is,
	there exist no non-trivial polynomial relations between $(\scrI_{\varsigma})_{\varsigma \in \Part(n)}$
	with coefficients in $\mathbb{C}$.
\end{theorem}

As a corollary,
we have a criterion for the triviality of $\scrI_{\Phi}$.
This is a generalization of the latter statement of~\cite{Marugame2021}*{Proposition 5.11}.

\begin{corollary}
\label{cor:triviality-of-Marugame's-invariant}
	For a $GL(n + 1, \mathbb{C})$-invariant polynomial $\Phi$ of degree at most $n$,
	the invariant $\scrI_{\Phi}$ is trivial if and only if $\Phi \equiv 0$ modulo $\ch_{1}$.
\end{corollary}

We will compute the Burns-Epstein invariant $\mu$ for Sasakian $\eta$-Einstein manifolds also.
Burns and Epstein~\cite{Burns-Epstein1990-Char} have introduced this invariant
for the boundaries of bounded strictly pseudoconvex domains in $\mathbb{C}^{n + 1}$.
Marugame~\cite{Marugame2016} has generalized this invariant for closed strictly pseudoconvex CR manifolds
admitting a pseudo-Einstein contact form.
He has defined this invariant as the boundary term of the renormalized Gauss-Bonnet-Chern formula
\begin{equation}
	\int_{\Omega} c_{n + 1}(\Theta)
	= \chi(\Omega) + \mu(M),
\end{equation}
where $\Omega$ is a strictly pseudoconvex domain bounded by $M$.

\begin{theorem}
\label{thm:Burns-Epstein-invariant-for-Sasakian-eta-Einstein-manifolds}
	Let $(S, T^{1, 0} S, \eta)$ be a closed $(2 n + 1)$-dimensional Sasakian $\eta$-Einstein manifold
	with Einstein constant $(n + 1) \lambda$.
	The Burns-Epstein invariant $\mu(S)$ of $S$ is given by
	\begin{equation}
		\mu(S)
		= \sum_{m = 0}^{n} \int_{S} \pqty{ - \frac{\lambda}{2 \pi} \eta }
			\wedge \pqty{ - \frac{\lambda}{2 \pi} d \eta }^{n - m}
			\wedge c_{m} \pqty{ \frac{\sqrt{- 1}}{2 \pi} \Omega },
	\end{equation}
	where $\Omega$ is the Tanaka-Webster curvature form.
\end{theorem}

Similar to \cref{thm:Marugame's-CR-invariants-of-Chern-character-type},
we obtain an expression of $\mu$ for the circle bundle case
in terms of the Einstein constant and characteristic numbers.

\begin{corollary}
\label{cor:Burns-Epstein-invariant-for-tubes}
	Let $(L, h)$ be a Hermitian holomorphic line bundle over a closed $n$-dimensional complex manifold $Y$
	such that $\omega = - \sqrt{- 1} \Theta_{h}$ defines a K\"{a}hler-Einstein metric on $Y$
	with Einstein constant $(n + 1) \lambda$.
	Denote by $(S, T^{1, 0} S, \eta)$ the circle bundle associated with $(Y, L, h)$.
	The Burns-Epstein invariant $\mu(S)$ of $S$ is given by
	\begin{equation}
		\mu(S)
		= - \lambda \sum_{m = 0}^{n}\int_{Y} (\lambda c_{1}(L))^{n - m} c_{m}(T^{1, 0} Y).
	\end{equation}
\end{corollary}

This paper is organized as follows.
In \cref{section:CR-geometry} (resp.\ \cref{section:strictly-pseudoconvex-domains}),
we recall basic facts on CR and Sasakian manifolds (resp.\ strictly pseudoconvex domains).
The renormalized connection and the invariants $\scrI_{\Phi}$ are introduced in \cref{section:global-CR-invariants-via-renormalized-characteristic-forms}.
\cref{section:global-CR-invariants-for-Sasakian-eta-Einstein-manifolds} is devoted to
proofs of \cref{thm:Marugame's-CR-invariant-for-Sasakian-eta-Einstein-manifolds,thm:Marugame's-CR-invariants-of-Chern-character-type,thm:algebraically-independence}.
\cref{section:Burns-Epstein-invariant-for-Sasakian-eta-Einstein-manifolds} deals with
the Burns-Epstein invariant.
In \cref{section:concluding-remarks},
we pose a conjecture on global CR invariants
and give a proof of it in low dimensions.

\medskip

\noindent
\emph{Notation.}
We use Einstein's summation convention and assume that 
\begin{itemize}
	\item lowercase Greek indices $\alpha, \beta, \gamma, \dots$ run from $1, \dots, n$;
	\item lowercase Latin indices $a, b, c, \dots$ run from $1, \dots , n, \infty$.
\end{itemize}

Suppose that a function $I(\epsilon)$ admits an asymptotic expansion,
as $\epsilon \to + 0$,
\begin{equation}
	I(\epsilon)
	=\sum_{m = 1}^{k} a_{m} \epsilon^{- m} + b \log \epsilon + O(1).
\end{equation}
Then the logarithmic part $\lp I(\epsilon)$ of $I(\epsilon)$ is the constant $b$.

\medskip

\section{CR geometry}
\label{section:CR-geometry}

\subsection{CR structures}
\label{subsection:CR-structures}

Let $M$ be a smooth $(2n + 1)$-dimensional manifold without boundary.
A \emph{CR structure} is a rank $n$ complex subbundle $T^{1,0}M$
of the complexified tangent bundle $TM \otimes \mathbb{C}$ such that
\begin{equation}
	T^{1, 0}M \cap T^{0, 1}M = 0, \qquad
	[\Gamma(T^{1, 0} M), \Gamma(T^{1, 0} M)] \subset \Gamma(T^{1, 0} M),
\end{equation}
where $T^{0, 1} M$ is the complex conjugate of $T^{1, 0} M$ in $T M \otimes \mathbb{C}$.
Set $H M = \Re T^{1,0}M$
and let $J \colon H M \to H M$ be the unique complex structure on $H M$ 
such that
\begin{equation}
	T^{1,0} M = \ker(J - \sqrt{- 1} \colon H M \otimes \mathbb{C}
	\to H M \otimes \mathbb{C}).
\end{equation}
A typical example of CR manifolds is a real hypersurface $M$ in an $(n + 1)$-dimensional complex manifold $X$;
this $M$ has the canonical CR structure
\begin{equation}
	T^{1, 0} M
	= T^{1, 0} X |_{M} \cap (T M \otimes \mathbb{C}).
\end{equation}

A CR structure $T^{1, 0} M$ is said to be \emph{strictly pseudoconvex}
if there exists a nowhere-vanishing real one-form $\theta$ on $M$
such that
$\theta$ annihilates $T^{1, 0} M$ and
\begin{equation}
	- \sqrt{- 1} d \theta (Z, \overline{Z}) > 0, \qquad
	0 \neq Z \in T^{1, 0} M.
\end{equation}
We call such a one-form a \emph{contact form}.
The triple $(M, T^{1, 0} M, \theta)$ is called a \emph{pseudo-Hermitian manifold}.
Denote by $T$ the \emph{Reeb vector field} with respect to $\theta$; 
that is,
the unique vector field satisfying
\begin{equation}
	\theta(T) = 1, \qquad T \contr d\theta = 0.
\end{equation}
Let $(\tensor{Z}{_{\alpha}})$ be a local frame of $T^{1, 0} M$,
and set $\tensor{Z}{_{\ovxa}} = \overline{\tensor{Z}{_{\alpha}}}$.
Then
$(T, \tensor{Z}{_{\alpha}}, \tensor{Z}{_{\ovxa}})$ gives a local frame of $T M \otimes \mathbb{C}$,
called an \emph{admissible frame}.
Its dual frame $(\theta, \tensor{\theta}{^{\alpha}}, \tensor{\theta}{^{\ovxa}})$
is called an \emph{admissible coframe}.
The two-form $d \theta$ is written as
\begin{equation}
	d \theta = \sqrt{- 1} \tensor{l}{_{\alpha}_{\ovxb}} \tensor{\theta}{^{\alpha}} \wedge \tensor{\theta}{^{\ovxb}},
\end{equation}
where $(\tensor{l}{_{\alpha}_{\ovxb}})$ is a positive definite Hermitian matrix.
We use $\tensor{l}{_{\alpha}_{\ovxb}}$ and its inverse $\tensor{l}{^{\alpha} ^{\ovxb}}$
to raise and lower indices of tensors.

\subsection{Tanaka-Webster connection and pseudo-Einstein condition}
\label{subsection:TW-connection-and-pE-condition}

A contact form $\theta$ induces a canonical connection $\nabla$,
called the \emph{Tanaka-Webster connection} with respect to $\theta$.
It is defined by
\begin{equation}
	\nabla T
	= 0,
	\quad
	\nabla \tensor{Z}{_{\alpha}}
	= \tensor{\omega}{_{\alpha}^{\beta}} \tensor{Z}{_{\beta}},
	\quad
	\nabla \tensor{Z}{_{\ovxa}}
	= \tensor{\omega}{_{\ovxa}^{\ovxb}} \tensor{Z}{_{\ovxb}}
	\quad
	\pqty{ \tensor{\omega}{_{\ovxa}^{\ovxb}}
	= \overline{\tensor{\omega}{_{\alpha}^{\beta}}} }
\end{equation}
with the following structure equations:
\begin{gather}
\label{eq:str-eq-of-TW-conn1}
	d \tensor{\theta}{^{\beta}}
	= \tensor{\theta}{^{\alpha}} \wedge \tensor{\omega}{_{\alpha}^{\beta}}
	+ \tensor{A}{^{\beta}_{\ovxa}} \theta \wedge \tensor{\theta}{^{\ovxa}}, \\
\label{eq:str-eq-of-TW-conn2}
	d \tensor{l}{_{\alpha}_{\ovxb}}
	= \tensor{\omega}{_{\alpha}^{\gamma}} \tensor{l}{_{\gamma}_{\ovxb}}
	+ \tensor{l}{_{\alpha}_{\ovxg}} \tensor{\omega}{_{\ovxb}^{\ovxg}}.
\end{gather}
The tensor $\tensor{A}{_{\alpha}_{\beta}} = \overline{\tensor{A}{_{\ovxa}_{\ovxb}}}$
is shown to be symmetric and is called the \emph{Tanaka-Webster torsion}.
We denote the components of a successive covariant derivative of a tensor
by subscripts preceded by a comma,
for example, $\tensor{K}{_{\alpha}_{\ovxb}_{,}_{\gamma}}$;
we omit the comma if the derivatives are applied to a function.
As we noted above,
we use $\tensor{l}{_{\alpha}_{\ovxb}}$ and its inverse $\tensor{l}{^{\alpha} ^{\ovxb}}$
to raise and lower indices of tensors;
for example,
$\tensor{K}{_{\alpha}_{\ovxb}_{,}^{\gamma}}
= \tensor{K}{_{\alpha}_{\ovxb}_{,}_{\overline{\rho}}} \tensor{l}{^{\gamma}^{\overline{\rho}}}$.

The curvature form
$\tensor{\Omega}{_{\alpha}^{\beta}} = d \tensor{\omega}{_{\alpha}^{\beta}}
- \tensor{\omega}{_{\alpha}^{\gamma}} \wedge \tensor{\omega}{_{\gamma}^{\beta}}$
of the Tanaka-Webster connection satisfies
\begin{equation}
\label{eq:curvature-form-of-TW-connection}
	\begin{split}
		\tensor{\Omega}{_{\alpha}^{\beta}}
		=& \tensor{R}{_{\alpha}^{\beta}_{\rho}_{\ovxs}} \tensor{\theta}{^{\rho}} \wedge \tensor{\theta}{^{\ovxs}}
			- \tensor{A}{_{\alpha}_{\gamma}_{,}^{\beta}} \theta \wedge \tensor{\theta}{^{\gamma}}
			+ \tensor{A}{^{\beta}_{\ovxg}_{,}_{\alpha}} \theta \wedge \tensor{\theta}{^{\ovxg}}  \\
		&- \sqrt{- 1} \tensor{A}{_{\alpha}_{\gamma}} \tensor{\theta}{^{\gamma}} \wedge \tensor{\theta}{^{\beta}}
			+ \sqrt{- 1} \tensor{l}{_{\alpha}_{\ovxg}} \tensor{A}{^{\beta}_{\ovxr}}
			\tensor{\theta}{^{\ovxg}} \wedge \tensor{\theta}{^{\ovxr}}.
	\end{split}
\end{equation}
We call the tensor $\tensor{R}{_{\alpha}^{\beta}_{\rho}_{\ovxs}}$
the \emph{Tanaka-Webster curvature}.
This tensor has the symmetry 
\begin{equation}
	\tensor{R}{_{\alpha}_{\ovxb}_{\rho}_{\ovxs}}
	= \tensor{R}{_{\rho}_{\ovxb}_{\alpha}_{\ovxs}}
	= \tensor{R}{_{\alpha}_{\ovxs}_{\rho}_{\ovxb}}.
\end{equation}
Contraction of indices gives the \emph{Tanaka-Webster Ricci curvature}
$\tensor{\Ric}{_{\rho}_{\ovxs}} = \tensor{R}{_{\alpha}^{\alpha}_{\rho}_{\ovxs}}$
and the \emph{Tanaka-Webster scalar curvature}
$\Scal = \tensor{\Ric}{_{\rho}^{\rho}}$.
The \emph{Chern tensor} $\tensor{S}{_{\alpha}_{\ovxb}_{\rho}_{\ovxs}}$ is defined by
\begin{equation}
	\tensor{S}{_{\alpha}_{\ovxb}_{\rho}_{\ovxs}}
	= \tensor{R}{_{\alpha}_{\ovxb}_{\rho}_{\ovxs}}
		- \tensor{P}{_{\alpha}_{\ovxb}} \tensor{l}{_{\rho}_{\ovxs}}
		- \tensor{P}{_{\rho}_{\ovxb}} \tensor{l}{_{\alpha}_{\ovxs}}
		- \tensor{P}{_{\rho}_{\ovxs}} \tensor{l}{_{\alpha}_{\ovxb}}
		- \tensor{P}{_{\alpha}_{\ovxs}} \tensor{l}{_{\rho}_{\ovxb}},
\end{equation}
where
\begin{equation}
	\tensor{P}{_{\alpha}_{\ovxb}}
	= \frac{1}{n + 2} \pqty{\tensor{\Ric}{_{\alpha}_{\ovxb}} - \frac{\Scal}{2 (n + 1)} \tensor{l}{_{\alpha}_{\ovxb}}}.
\end{equation}
A contact form $\theta$ is said to be \emph{pseudo-Einstein}
if the following two equalities hold:
\begin{equation}
\label{eq:pseudo-Einstein-condition}
	\tensor{\Ric}{_{\alpha}_{\overline{\beta}}}
	= \frac{1}{n} \Scal \cdot \tensor{l}{_{\alpha}_{\overline{\beta}}},
	\qquad
	\tensor{\Scal}{_{\alpha}}
	= \sqrt{- 1} n \tensor{A}{_{\alpha}_{\beta}_{,}^{\beta}}.
\end{equation}
From Bianchi identities for the Tanaka-Webster connection,
we obtain
\begin{equation}
	\tensor{\pqty{\tensor{\Ric}{_{\alpha}_{\overline{\beta}}} - \frac{1}{n} \Scal \cdot \tensor{l}{_{\alpha}_{\overline{\beta}}}}}
	{_{,}^{\overline{\beta}}}
	= \frac{n - 1}{n} (\tensor{\Scal}{_{\alpha}} - \sqrt{- 1} n \tensor{A}{_{\alpha}_{\beta}_{,}^{\beta}});
\end{equation}
see~\cite{Hirachi2014}*{Lemma 5.7(iii)} for example.
Hence the latter equality of \cref{eq:pseudo-Einstein-condition} follows from the former one if $n \geq 2$.
On the other hand,
the former equality of \cref{eq:pseudo-Einstein-condition} automatically holds if $n = 1$,
and the latter one is a non-trivial condition.

\subsection{Sasakian manifolds}
\label{subsection:Sasakian-manifolds}

Sasakian manifolds are an important class of pseudo-Hermitian manifolds.
See~\cite{Boyer-Galicki2008} for a comprehensive introduction to Sasakian manifolds.

A \emph{Sasakian manifold} is a pseudo-Hermitian manifold $(S, T^{1,0}S, \eta)$
with vanishing Tanaka-Webster torsion.
This condition is equivalent to that the Reeb vector field $\xi$ with respect to $\eta$
preserves the CR structure $T^{1, 0} S$.
An almost complex structure $I$ on the cone $C(S) = \mathbb{R}_{+} \times S$ of $S$
is defined by
\begin{equation}
	I(a (r \pdvf{}{r}) + b \xi + V) = - b (r \pdvf{}{r}) + a \xi + J V,
\end{equation}
where $r$ is the coordinate of $\mathbb{R}_{+}$,
$a, b \in \mathbb{R}$, and $V \in H S$.
The bundle $T^{1, 0} C(S)$ of $(1, 0)$-vectors with respect to $I$
is given by
\begin{equation}
	T^{1, 0} C(S)
	= \mathbb{C}(r \pdvf{}{r} - \sqrt{- 1} \xi) \oplus T^{1, 0} S.
\end{equation}
The vanishing of the Tanaka-Webster torsion implies that
$I$ is integrable;
that is,
$(C(S), I)$ is a complex manifold.
Moreover,
the one-form $\eta$ is equal to $d^{c} \log r^{2}$.
In what follows,
we identify $S$ with the level set $\{1\} \times S \subset C(S)$.

A $(2 n + 1)$-dimensional Sasakian manifold $(S, T^{1, 0} S, \eta)$
is said to be \emph{Sasakian $\eta$-Einstein with Einstein constant $(n + 1) \lambda$}
if the Tanaka-Webster Ricci curvature satisfies
\begin{equation}
	\tensor{\Ric}{_{\alpha}_{\ovxb}}
	= (n + 1) \lambda \tensor{l}{_{\alpha}_{\ovxb}}.
\end{equation}
Note that $\eta$ is a pseudo-Einstein contact form on $S$.

A typical example of Sasakian manifolds
is the circle bundle associated with a negative Hermitian line bundle.
Let $Y$ be an $n$-dimensional complex manifold
and $(L, h)$ a Hermitian holomorphic line bundle over $Y$
such that
\begin{equation}
	\omega = - \sqrt{-1} \Theta_{h} = d d^{c} \log h
\end{equation}
is a K\"{a}hler form on $Y$.
Note that $c_{1}(L) = - [\omega / 2 \pi]$.
Consider the circle bundle
\begin{equation}
	S = \Set{ v \in L | h(v, v) = 1}
\end{equation}
over $Y$,
which is a real hypersurface in the total space of $L$.
The one-form $\eta = d^{c} \log h |_{S}$
is a connection one-form of the principal $S^{1}$-bundle $p \colon S \to Y$
and satisfies $d \eta = p^{\ast} \omega$.
Moreover,
the natural CR structure $T^{1, 0} S$ coincides with the horizontal lift of $T^{1, 0} Y$ with respect to $\eta$.
Since $\omega$ is a K\"{a}hler form,
we have
\begin{equation}
	- \sqrt{- 1} d \eta (Z, \overline{Z})
	= - \sqrt{- 1} \omega (p_{\ast} Z, p_{\ast} \overline{Z})
	> 0
\end{equation}
for all non-zero $Z \in T^{1, 0} S$.
This implies that $(S, T^{1, 0} S)$ is a strictly pseudoconvex CR manifold
and $\eta$ is a contact form on $S$.
Note that the Reeb vector field $\xi$ with respect to $\eta$
is the generator of the $S^{1}$-action on $S$.

Consider the Tanaka-Webster connection with respect to $\eta$.
Take a local coordinate $(z^{1}, \dots , z^{n})$ of $Y$.
The K\"{a}hler form $\omega$ is written as
\begin{equation}
	\omega
	= \sqrt{- 1} \tensor{g}{_{\alpha}_{\ovxb}} d z^{\alpha} \wedge d \ovz^{\beta},
\end{equation}
where $(\tensor{g}{_{\alpha}_{\ovxb}})$ is a positive definite Hermitian matrix.
Let $Z_{\alpha}$ be the horizontal lift of $\pdvf{}{z^{\alpha}}$.
Then $(\xi, Z_{\alpha}, Z_{\ovxa} = \overline{Z_{\alpha}})$
is an admissible frame on $S$.
The corresponding admissible coframe is given by
$(\eta, \theta^{\alpha} = p^{\ast} (d z^{\alpha}), \theta^{\ovxa} = p^{\ast} (d \ovz^{\alpha}))$.
Since $d \eta = p^{*} \omega$,
we have
\begin{equation}
	d \eta = \sqrt{- 1} (p^{\ast} \tensor{g}{_{\alpha}_{\ovxb}}) \theta^{\alpha} \wedge \theta^{\ovxb},
\end{equation}
which implies $\tensor{l}{_{\alpha}_{\ovxb}} = p^{\ast} \tensor{g}{_{\alpha}_{\ovxb}}$.
The connection form $\tensor{\pi}{_{\alpha}^{\beta}}$ of the K\"{a}hler metric
with respect to the frame $(\pdvf{}{\tensor{z}{^{\alpha}}})$ satisfies
\begin{equation} \label{eq:structure-equation-for-Kahler-metric}
	0 = d (d z^{\beta}) = d \tensor{z}{^{\alpha}} \wedge \tensor{\pi}{_{\alpha}^{\beta}},
	\qquad
	d \tensor{g}{_{\alpha}_{\ovxb}}
	= \tensor{\pi}{_{\alpha}^{\gamma}} \tensor{g}{_{\gamma}_{\ovxb}}
	+ \tensor{g}{_{\alpha}_{\ovxg}} \tensor{\pi}{_{\ovxb}^{\ovxg}}
	\qquad
	\pqty{\tensor{\pi}{_{\ovxa}^{\ovxb}} = \overline{\tensor{\pi}{_{\alpha}^{\beta}}}}.
\end{equation}
We write as $\tensor{\Pi}{_{\alpha}^{\beta}}$ the curvature form of the K\"{a}hler metric.
Pulling back \cref{eq:structure-equation-for-Kahler-metric} by $p$ gives
\begin{equation}
	d \theta^{\beta} = \tensor{\theta}{^{\alpha}} \wedge (p^{*} \tensor{\pi}{_{\alpha}^{\beta}}),
	\qquad
	d \tensor{l}{_{\alpha}_{\ovxb}}
	= (p^{\ast} \tensor{\pi}{_{\alpha}^{\gamma}}) \tensor{l}{_{\gamma}_{\ovxb}}
	+ \tensor{l}{_{\alpha}_{\ovxg}} (p^{\ast} \tensor{\pi}{_{\ovxb}^{\ovxg}}).
\end{equation}
This yields that
$\tensor{\omega}{_{\alpha}^{\beta}} = p^{\ast} \tensor{\pi}{_{\alpha}^{\beta}}$,
and the Tanaka-Webster torsion vanishes identically;
that is,
$(S, T^{1, 0} S, \eta)$ is a Sasakian manifold.
Moreover,
the curvature form $\tensor{\Omega}{_{\alpha}^{\beta}}$ of the Tanaka-Webster connection
is given by $\tensor{\Omega}{_{\alpha}^{\beta}} = p^{\ast} \tensor{\Pi}{_{\alpha}^{\beta}}$.
In particular,
$(S, T^{1, 0} S, \eta)$ is a Sasakian $\eta$-Einstein manifold with Einstein constant $(n + 1) \lambda$
if and only if $\omega$ defines a K\"{a}hler-Einstein metric with Einstein constant $(n + 1) \lambda$.

\section{Strictly pseudoconvex domains}
\label{section:strictly-pseudoconvex-domains}

Let $\Omega$ be a relatively compact domain in an $(n + 1)$-dimensional complex manifold $X$
with smooth boundary $M = \bdry \Omega$.
There exists a smooth function $\rho$ on $X$ such that
\begin{equation}
	\Omega = \rho^{-1}((- \infty, 0)), \quad
	M = \rho^{- 1}(0), \quad
	d \rho \neq 0 \ \text{on} \ M;
\end{equation}
such a $\rho$ is called a \emph{defining function} of $\Omega$.
A domain $\Omega$ is said to be \emph{strictly pseudoconvex}
if we can take a defining function $\rho$ of $\Omega$ that is strictly plurisubharmonic near $M$.
The boundary of a strictly pseudoconvex domain
is a closed strictly pseudoconvex real hypersurface.
Conversely,
it is known that
any closed connected strictly pseudoconvex CR manifold of dimension at least five can be realized as
the boundary of a strictly pseudoconvex domain
in a complex projective manifold~\cites{Boutet_de_Monvel1975,Harvey-Lawson1975,Lempert1995}.

\subsection{Graham-Lee connection}
\label{subsection:Graham-Lee-connection}

Let $\rho$ be a defining function of a strictly pseudoconvex domain $\Omega$
in an $(n + 1)$-dimensional complex manifold $X$.
There exists the unique $(1, 0)$-vector field $\tensor{\wtZ}{_{\infty}}$ near the boundary such that
\begin{equation}
	\tensor{\wtZ}{_{\infty}} \rho = 1,
	\qquad
	\tensor{\wtZ}{_{\infty}} \contr \del \delb \rho
	= \kappa \delb \rho
\end{equation}
for a smooth function $\kappa$,
which is called the \emph{transverse curvature}.
Take a local frame $(\tensor{\wtZ}{_{\alpha}})$ of $\Ker \del \rho$,
and let $(\tensor{\wtxth}{^{\alpha}}, \tensor{\wtxth}{^{\infty}} = \del \rho)$ be
the dual frame of $(\tensor{\wtZ}{_{\alpha}}, \tensor{\wtZ}{_{\infty}})$.
By the strict pseudoconvexity,
there exists a positive Hermitian matrix $(\tensor{\wtl}{_{\alpha}_{\ovxb}})$ such that
\begin{equation}
	d d^{c} \rho
	= \sqrt{- 1} \tensor{\wtl}{_{\alpha}_{\ovxb}} \tensor{\wtxth}{^{\alpha}} \wedge \tensor{\wtxth}{^{\ovxb}}
		+ \kappa d \rho \wedge d^{c} \rho.
\end{equation}
We use $\tensor{\wtl}{_{\alpha}_{\ovxb}}$ and its inverse $\tensor{\wtl}{^{\alpha} ^{\ovxb}}$
to raise and lower indices of tensors.
The \emph{Graham-Lee connection} $\widetilde{\nabla}$
is the unique connection on $T X$ defined by
\begin{equation}
	\widetilde{\nabla} \tensor{\wtZ}{_{\alpha}}
	= \tensor{\wtxo}{_{\alpha}^{\beta}} \tensor{\wtZ}{_{\beta}},
	\quad
	\widetilde{\nabla} \tensor{\wtZ}{_{\ovxa}}
	= \tensor{\wtxo}{_{\ovxa}^{\ovxb}} \tensor{\wtZ}{_{\ovxb}}
	\quad
	\pqty{ \tensor{\wtxo}{_{\ovxa}^{\ovxb}}
	= \overline{\tensor{\wtxo}{_{\alpha}^{\beta}}} },
	\quad
	\widetilde{\nabla} \tensor{\wtZ}{_{\infty}}
	= \widetilde{\nabla} \tensor{\wtZ}{_{\overline{\infty}}}
	= 0,
\end{equation}
with the following structure equations:
\begin{gather}
	d \tensor{\wtxth}{^{\beta}}
	= \tensor{\wtxth}{^{\alpha}} \wedge \tensor{\wtxo}{_{\alpha}^{\beta}}
		- \sqrt{- 1} \tensor{\wtA}{^{\beta}_{\ovxs}} \del \rho \wedge \tensor{\wtxth}{^{\ovxs}}
		- \tensor{\kappa}{^{\beta}} \del \rho \wedge \delb \rho
		+ \frac{1}{2} \kappa d \rho \wedge \tensor{\wtxth}{^{\beta}}, \\
	d \tensor{\wtl}{_{\alpha}_{\ovxb}}
	= \tensor{\wtxo}{_{\alpha}^{\gamma}} \tensor{\wtl}{_{\gamma}_{\ovxb}}
	+ \tensor{\wtl}{_{\alpha}_{\ovxg}} \tensor{\wtxo}{_{\ovxb}^{\ovxg}};
\end{gather}
see~\cite{Graham-Lee1988}*{Proposition 1.1} for a proof of the existence and uniqueness.
Note that the restriction of $\widetilde{\nabla}$ to $M$
coincides with the Tanaka-Webster connection with respect to $\theta = d^{c} \rho |_{M}$.

\subsection{Fefferman defining functions}
\label{subsection:Fefferman-defining-function}

Let $\Omega$ be a strictly pseudoconvex domain
in a complex manifold $X$ of dimension $n + 1$.
Fix a local coordinate $z$ on $X$
near a point of the boundary $M = \bdry \Omega$.
A differential operator $\mathcal{J}_{z}$ is defined by
\begin{equation}
	\mathcal{J}_{z}[\phi]
	= - \det
	\begin{pmatrix}
		\phi & \partial \phi / \partial z^{i} \\
		\partial \phi / \partial \ovz^{j} & \partial^{2} \phi / \partial z^{i} \partial \ovz^{j}
	\end{pmatrix}
	.
\end{equation}
A \emph{Fefferman defining function}
is a defining function of $\Omega$ such that
\begin{equation}
	d d^{c} \log \mathcal{J}_{z}[\rho]
	= d d^{c} O(\rho^{n + 2}).
\end{equation}
This condition is independent of the choice of a local coordinate $z$.
It is known that $\Omega$ admits a Fefferman defining function
if and only if the boundary $M$ has a pseudo-Einstein contact form;
see~\cite{Lee1988}*{Theorem 4.2}, \cite{Hirachi1993}*{Lemma 7.2}, and \cite{Hislop-Perry-Tang2008}*{Proposition 2.10}.

\section{Global CR invariants via renormalized characteristic forms}
\label{section:global-CR-invariants-via-renormalized-characteristic-forms}

\subsection{Burns-Epstein's renormalized connection}
\label{subsection:renormalized-connection}

Let $\Omega$ be a strictly pseudoconvex domain of dimension $n + 1$.
Assume that its boundary $M$ admits a pseudo-Einstein contact form.
For a Fefferman defining function $\rho$ of $\Omega$,
the $(1, 1)$-form
\begin{equation}
	\omega_{+}
	= - d d^{c} \log (- \rho)
\end{equation}
defines a K\"{a}hler metric $g_{+}$ near the boundary.
We extend $g_{+}$ to a Hermitian metric on $\Omega$.
Let $\tensor{\psi}{_{a}^{b}}$ be the Chern connection with respect to $g_{+}$.
This connection diverges on the boundary since so does $g_{+}$.
The \emph{renormalized connection form} $\tensor{\theta}{_{a}^{b}}$ is defined by
\begin{equation}
	\tensor{\theta}{_{a}^{b}}
	= \tensor{\psi}{_{a}^{b}}
		+ \frac{1}{\rho} (\tensor{\delta}{_{a}^{b}} \tensor{\rho}{_{c}}
		+ \tensor{\delta}{_{c}^{b}} \tensor{\rho}{_{a}}) \tensor{\wtxth}{^{c}}. 
\end{equation}
This connection form can be extended smoothly up to the boundary;
see~\cite{Marugame2021}*{Proposition 4.1} for example.
The corresponding curvature form is denoted by $\tensor{\Theta}{_{a}^{b}}$,
which satisfies
\begin{equation}
\label{eq:trace-of-renormalized-curvature}
	\tr \Theta
	= - d d^{c} \log \mathcal{J}_{z}[\rho]
	= d d^{c} O(\rho^{n + 2});
\end{equation}
see~\cite{Marugame2016}*{(4.7)}.
Near the boundary,
$\tensor{\theta}{_{a}^{b}}$ is written in terms of the Graham-Lee connection.

\begin{lemma}[\cite{Marugame2016}*{Proposition 3.5}]
\label{lem:renormalized-connection}
	For the frame $(\tensor{\wtxth}{^{\alpha}}, \wtxth^{\infty} = \partial \rho)$,
	\begin{gather}
		\tensor{\theta}{_{\alpha}^{\beta}}
		= \tensor{\wtxo}{_{\alpha}^{\beta}}
			+ \frac{1}{2} \kappa (\del \rho - \delb \rho) \tensor{\delta}{_{\alpha}^{\beta}}, \\
		\tensor{\theta}{_{\infty}^{\beta}}
		= \kappa \tensor{\wtxth}{^{\beta}}
			- \sqrt{- 1} \tensor{\wtA}{^{\beta}_{\ovxg}} \tensor{\wtxth}{^{\ovxg}}
			- \tensor{\kappa}{^{\beta}} \delb \rho, \\
		\tensor{\theta}{_{\alpha}^{\infty}}
		= - \tensor{\wtl}{_{\alpha}_{\ovxg}} \tensor{\wtxth}{^{\ovxg}}
			- \rho (1 - \kappa \rho)^{- 1} \tensor{\kappa}{_{\alpha}} \del \rho
			+ \sqrt{- 1} \rho (1 - \kappa \rho)^{- 1} \tensor{\wtA}{_{\alpha}_{\beta}} \tensor{\wtxth}{^{\beta}}, \\
		\tensor{\theta}{_{\infty}^{\infty}}
		= - \kappa \delb \rho
			- \rho \kappa^{2} (1 - \kappa \rho)^{- 1} \del \rho
			- \rho (1 - \kappa \rho)^{- 1} \del \kappa.
	\end{gather}
\end{lemma}

\subsection{Global CR invariants via renormalized characteristic forms}

Let $\Omega$ be an $(n + 1)$-dimensional strictly pseudoconvex domain with boundary $M$.
Assume that $M$ admits a pseudo-Einstein contact form,
and take a Fefferman defining function $\rho$ of $\Omega$.

\begin{definition}[\cite{Marugame2021}*{Theorem 1.1 and Proposition 4.9}]
	Let $\Phi$ be a $GL(n + 1, \mathbb{C})$-invariant polynomial of degree at most $n$.
	Then
	\begin{equation}
		\scrI_{\Phi}(M)
		= - \sum_{m = 0}^{n} \lp \int_{\rho < - \epsilon}
			\frac{d \log (- \rho) \wedge d^{c} \log (- \rho)}{2 \pi}
			\wedge \pqty{ \frac{\omega_{+}}{2 \pi} }^{n - m}
			\wedge \Phi \pqty{ \frac{\sqrt{- 1}}{2 \pi} \Theta }
	\end{equation}
	is independent of the choice of $\rho$,
	and gives a global CR invariant of $M$.
\end{definition}

In the remainder of this section,
we show that any $\scrI_{\Phi}$ is written as a linear combination of $\scrI_{\varsigma}$,
which appears in the introduction.
It is known that $(\ch_{k})_{k = 1}^{n + 1}$ generate
the algebra of $GL(n + 1, \mathbb{C})$-invariant polynomials
and are algebraically independent over $\mathbb{C}$.
For a $GL(n + 1, \mathbb{C})$-invariant polynomial $\Phi$ of degree at most $n$,
there exists a $GL(n + 1, \mathbb{C})$-invariant polynomial $\wtxcph$ of degree at most $n - 1$
and a family $(C^{\Phi}_{\varsigma})_{\varsigma \in \Part(n)}$ of complex numbers such that
\begin{equation}
	\Phi
	= \ch_{1} \wtxcph + \sum_{\varsigma \in \Part(n)} C^{\Phi}_{\varsigma} \Phi_{\varsigma},
\end{equation}
where $\Phi_{\varsigma}$ is defined by \cref{eq:generator-of-Marugame's-invariant}.
Note that $\wtxcph$ and $C^{\Phi}_{\varsigma}$ are unique.

\begin{lemma}
\label{lem:linear-combination-of-Marugame's-invariant}
	For a $GL(n + 1, \mathbb{C})$-invariant polynomial $\Phi$ of degree at most $n$,
	\begin{equation}
		\scrI_{\Phi}
		= \sum_{\varsigma \in \Part(n)} C^{\Phi}_{\varsigma} \scrI_{\varsigma}.
	\end{equation}
	In particular,
	$\scrI_{\Phi}$ is trivial if $\Phi \equiv 0$ modulo $\ch_{1}$.
\end{lemma}

\begin{proof}
	It suffices to show that $\scrI_{\Phi}$ is trivial if $\Phi$ is of the form $\ch_{1} \wtxcph$.
	\cref{eq:trace-of-renormalized-curvature} implies that
	\begin{equation}
		\ch_{1}(\Theta)
		= O(\rho^{n + 1})
		\qquad
		\text{mod $d \rho$}.
	\end{equation}
	Hence,
	for $1 \leq m \leq n$,
	\begin{align}
		&\frac{d \log (- \rho) \wedge d^{c} \log (- \rho)}{2 \pi}
			\wedge \pqty{ \frac{\omega_{+}}{2 \pi} }^{n - m}
			\wedge \Phi \pqty{ \frac{\sqrt{- 1}}{2 \pi} \Theta } \\
		&=  \frac{d \rho \wedge d^{c} \rho}{2 \pi \rho^{2}}
			\wedge \pqty{ \frac{d d^{c} \rho}{2 \pi \rho} }^{n - m}
			\wedge \ch_{1} \pqty{ \frac{\sqrt{- 1}}{2 \pi} \Theta }
			\wedge \wtxcph  \pqty{ \frac{\sqrt{- 1}}{2 \pi} \Theta } \\
		&= O(1).
	\end{align}
	Therefore the integral
	\begin{equation}
		\int_{\rho < - \epsilon}
			\frac{d \log (- \rho) \wedge d^{c} \log (- \rho)}{2 \pi}
			\wedge \pqty{ \frac{\omega_{+}}{2 \pi} }^{n - m}
			\wedge \Phi \pqty{ \frac{\sqrt{- 1}}{2 \pi} \Theta }
	\end{equation}
	has a finite limit as $\epsilon \to + 0$.
	On the other hand,
	since $\Omega$ is of complex dimension $n + 1$,
	the differential form
	\begin{equation}
		\frac{d \log (- \rho) \wedge d^{c} \log (- \rho)}{2 \pi}
			\wedge \pqty{ \frac{\omega_{+}}{2 \pi} }^{n}
			\wedge \Phi \pqty{ \frac{\sqrt{- 1}}{2 \pi} \Theta }
	\end{equation}
	is identically zero.
	Therefore $\scrI_{\Phi} = 0$.
\end{proof}

\section{Global CR invariants for Sasakian $\eta$-Einstein manifolds}
\label{section:global-CR-invariants-for-Sasakian-eta-Einstein-manifolds}

Let $(S, T^{1, 0} S, \eta)$ be a $(2 n + 1)$-dimensional Sasakian $\eta$-Einstein manifold
with Einstein constant $(n + 1) \lambda$.
Then
\begin{equation}
\label{eq:Fefferman-defining-function-of-Sasakian-eta-Einstein}
	\rho
	=
	\begin{cases}
		\lambda^{- 1} (r^{2 \lambda} - 1) & \lambda \neq 0, \\
		\log r^{2} & \lambda = 0,
	\end{cases}
\end{equation}
is a Fefferman defining function of $\Set{r < 1}$ in $C(S)$~\cite{Takeuchi2018}*{Proposition 3.1}.
Note that
\begin{equation}
	d \rho
	= (1 + \lambda \rho) d \log r^{2},
	\qquad
	d^{c} \rho
	= (1 + \lambda \rho) \eta.
\end{equation}
Let $(\eta, \tensor{\theta}{^{\alpha}}, \tensor{\theta}{^{\ovxa}})$ be an admissible coframe on $S$.
Then
\begin{equation}
	d d^{c} \rho
	= \sqrt{- 1}(1 + \lambda \rho) \tensor{l}{_{\alpha}_{\ovxb}} \tensor{\theta}{^{\alpha}} \wedge \tensor{\theta}{^{\ovxb}}
		+ \lambda (1 + \lambda \rho)^{- 1} d \rho \wedge d^{c} \rho.
\end{equation}
In particular,
\begin{equation}
\label{eq:Levi-form-and-transverse-curvature-for-Sasakian-eta-Einstein}
	\tensor{\wtl}{_{\alpha}_{\ovxb}} =
	(1 + \lambda \rho) \tensor{l}{_{\alpha}_{\ovxb}},
	\qquad
	\kappa = \lambda (1 + \lambda \rho)^{- 1}.
\end{equation}
We compute the Graham-Lee connection with respect to $\rho$.
\cref{eq:str-eq-of-TW-conn1,eq:str-eq-of-TW-conn2} yield that
\begin{align}
	d \tensor{\theta}{^{\beta}}
	&= \tensor{\theta}{^{\alpha}} \wedge \tensor{\omega}{_{\alpha}^{\beta}}
	= \tensor{\theta}{^{\alpha}} \wedge \pqty{ \tensor{\omega}{_{\alpha}^{\beta} }
		+ \frac{1}{2} \lambda (d \log r^{2}) \tensor{\delta}{_{\alpha}^{\beta}}}
		+ \frac{1}{2} \lambda (1 + \lambda \rho)^{- 1} d \rho \wedge \tensor{\theta}{^{\beta}}, \\
	d \tensor{\wtl}{_{\alpha}_{\ovxb}}
	&= \lambda (1 + \lambda \rho) d \log r^{2} \cdot \tensor{l}{_{\alpha}_{\ovxb}}
		+ (1 + \lambda \rho) d \tensor{l}{_{\alpha}_{\ovxb}} \\
	&= \pqty{ \tensor{\omega}{_{\alpha}^{\gamma}} + \frac{1}{2} \lambda (d \log r^{2}) \tensor{\delta}{_{\alpha}^{\gamma}} } \tensor{\wtl}{_{\gamma}_{\ovxb}}
		+ \tensor{\wtl}{_{\alpha}_{\ovxg}} \pqty{ \tensor{\omega}{_{\ovxb}^{\ovxg}} + \frac{1}{2} \lambda (d \log r^{2}) \tensor{\delta}{_{\ovxb}^{\ovxg}} }.
\end{align}
Hence the uniqueness of the Graham-Lee connection implies
\begin{equation}
\label{eq:GL-connection-and-torsion-for-Sasakian-eta-Einstein}
	\tensor{\wtxo}{_{\alpha}^{\beta}}
	= \tensor{\omega}{_{\alpha}^{\beta}} + \frac{1}{2} \lambda (d \log r^{2}) \tensor{\delta}{_{\alpha}^{\beta}},
	\qquad
	\tensor{\wtA}{_{\alpha}_{\beta}} = 0.
\end{equation}

\begin{lemma}
\label{lem:renormalized-connection-for-Sasakian-eta-Einstein}
	Let $(S, T^{1, 0} S, \eta)$ be a $(2 n + 1)$-dimensional Sasakian $\eta$-Einstein manifold with Einstein constant $(n + 1) \lambda$.
	For a Fefferman defining function $\rho$ given by \cref{eq:Fefferman-defining-function-of-Sasakian-eta-Einstein},
	the renormalized connection and curvature satisfy
	\begin{gather}
		\tensor{\theta}{_{\alpha}^{\beta}}
		= \tensor{\omega}{_{\alpha}^{\beta}} + \lambda (\del \log r^{2}) \tensor{\delta}{_{\alpha}^{\beta}},
		\qquad
		\tensor{\theta}{_{\infty}^{\beta}}
		= \lambda (1 + \lambda \rho)^{- 1} \tensor{\theta}{^{\beta}}, \\
		\tensor{\theta}{_{\alpha}^{\infty}}
		= - (1 + \lambda \rho) \tensor{l}{_{\alpha}_{\ovxg}} \tensor{\theta}{^{\ovxg}},
		\qquad
		\tensor{\theta}{_{\infty}^{\infty}}
		= - \lambda \delb \log r^{2}, \\
		\tensor{\Theta}{_{\alpha}^{\beta}}
		= \tensor{\Omega}{_{\alpha}^{\beta}}
			+ \sqrt{- 1} \lambda d \eta \cdot \tensor{\delta}{_{\alpha}^{\beta}}
			- \lambda \tensor{l}{_{\alpha}_{\overline{\gamma}}} \tensor{\theta}{^{\beta}} \wedge \tensor{\theta}{^{\overline{\gamma}}}
		= \tensor{S}{_{\alpha}^{\beta}_{\rho}_{\ovxs}} \tensor{\theta}{^{\rho}} \wedge \tensor{\theta}{^{\ovxs}}, \\
		\tensor{\Theta}{_{\infty}^{\beta}}
		= 0,
		\qquad
		\tensor{\Theta}{_{\alpha}^{\infty}}
		= 0,
		\qquad
		\tensor{\Theta}{_{\infty}^{\infty}}
		= 0.
	\end{gather}
\end{lemma}

\begin{proof}
	The equalities for $\tensor{\theta}{_{a}^{b}}$ are consequences of
	\cref{lem:renormalized-connection,eq:Levi-form-and-transverse-curvature-for-Sasakian-eta-Einstein,eq:GL-connection-and-torsion-for-Sasakian-eta-Einstein}.
	The curvature form $\Theta$ is defined by
	\begin{equation}
		\tensor{\Theta}{_{a}^{b}}
		= d \tensor{\theta}{_{a}^{b}} - \tensor{\theta}{_{a}^{c}} \wedge \tensor{\theta}{_{c}^{b}}.
	\end{equation}
	Hence \cref{eq:curvature-form-of-TW-connection} implies
	\begin{align}
		\tensor{\Theta}{_{\alpha}^{\beta}}
		&= d \tensor{\omega}{_{\alpha}^{\beta}} + \lambda (\delb \del \log r^{2}) \tensor{\delta}{_{\alpha}^{\beta}} \\
		& \qquad - (\tensor{\omega}{_{\alpha}^{\gamma}} + \lambda (\del \log r^{2}) \tensor{\delta}{_{\alpha}^{\gamma}})
			\wedge (\tensor{\omega}{_{\gamma}^{\beta}} + \lambda (\del \log r^{2}) \tensor{\delta}{_{\gamma}^{\beta}}) \\
		& \qquad + \lambda \tensor{l}{_{\alpha}_{\ovxg}} \tensor{\theta}{^{\ovxg}} \wedge \tensor{\theta}{^{\beta}} \\
		&= \tensor{\Omega}{_{\alpha}^{\beta}}
			+ \sqrt{- 1} \lambda d \eta \cdot \tensor{\delta}{_{\alpha}^{\beta}}
			- \lambda \tensor{l}{_{\alpha}_{\ovxg}} \tensor{\theta}{^{\beta}} \wedge \tensor{\theta}{^{\ovxg}} \\
		&= \tensor{S}{_{\alpha}^{\beta}_{\rho}_{\ovxs}} \tensor{\theta}{^{\rho}} \wedge \tensor{\theta}{^{\ovxs}}.
	\end{align}
	Similarly,
	\cref{eq:str-eq-of-TW-conn1} yields
	\begin{align}
		\tensor{\Theta}{_{\infty}^{\beta}}
		&= - \lambda^{2} (1 + \lambda \rho)^{- 1} (d \log r^{2}) \wedge \tensor{\theta}{^{\beta}} 
			+ \lambda (1 + \lambda \rho)^{- 1} d \theta^{\beta} \\
		& \quad - \lambda (1 + \lambda \rho)^{- 1} \tensor{\theta}{^{\gamma}}
			\wedge (\tensor{\omega}{_{\gamma}^{\beta}} + \lambda (\del \log r^{2}) \tensor{\delta}{_{\gamma}^{\beta}}) \\
		& \quad + \lambda^{2} (1 + \lambda \rho)^{- 1} (\delb \log r^{2}) \wedge \tensor{\theta}{^{\beta}} \\
		&= 0.
	\end{align}
	By using both \cref{eq:str-eq-of-TW-conn1,eq:str-eq-of-TW-conn2},
	we have
	\begin{align}
		\tensor{\Theta}{_{\alpha}^{\infty}}
		&= - \lambda (1 + \lambda \rho) \tensor{l}{_{\alpha}_{\ovxg}} (d \log r^{2}) \wedge \tensor{\theta}{^{\ovxg}}
			- (1 + \lambda \rho) d \tensor{l}{_{\alpha}_{\ovxg}} \wedge \tensor{\theta}{^{\ovxg}}
			- (1 + \lambda \rho) \tensor{l}{_{\alpha}_{\ovxg}} d \tensor{\theta}{^{\ovxg}} \\
		& \quad + (1 + \lambda \rho) \tensor{l}{_{\beta}_{\ovxg}}
			(\tensor{\omega}{_{\alpha}^{\beta}} + \lambda (\del \log r^{2}) \tensor{\delta}{_{\alpha}^{\beta}})
			\wedge \tensor{\theta}{^{\ovxg}} \\
		& \quad - \lambda (1 + \lambda \rho) \tensor{l}{_{\alpha}_{\ovxg}} \tensor{\theta}{^{\ovxg}} \wedge (\delb \log r^{2}) \\
		&= 0.
	\end{align}
	Finally,
	\begin{equation}
		\tensor{\Theta}{_{\infty}^{\infty}}
		= - \lambda \del \delb \log r^{2} + \lambda \tensor{l}{_{\beta}_{\ovxg}}
			\tensor{\theta}{^{\beta}} \wedge \tensor{\theta}{^{\ovxg}}
		= 0.
	\end{equation}
	These complete the proof.
\end{proof}

\begin{proof}[Proof of \cref{thm:Marugame's-CR-invariant-for-Sasakian-eta-Einstein-manifolds}]
	From the definition of $\mathscr{I}_{\Phi}$ and \cref{lem:renormalized-connection-for-Sasakian-eta-Einstein},
	it follows that
	\begin{align}
		&\scrI_{\Phi}(S) \\
		&= - \sum_{m = 0}^{n} \lp \int_{\rho < - \epsilon} \frac{(1 + \lambda \rho)^{n - m + 1}}{(- \rho)^{n - m + 2}}
			d \rho \wedge \pqty{ \frac{\eta}{2 \pi} } \wedge \pqty{ \frac{d \eta}{2 \pi} }^{n - m}
			\wedge \Phi \pqty{ \frac{\sqrt{-1}}{2 \pi} \Theta } \\
		&= - \sum_{m = 0}^{n} \pqty{ \lp \int_{\bullet}^{- \epsilon} \frac{(1 + \lambda \rho)^{n - m + 1}}{(- \rho)^{n - m + 2}} d \rho}
			\times \int_{S} \pqty{ \frac{\eta}{2 \pi} } \wedge \pqty{ \frac{d \eta}{2 \pi} }^{n - m}
			\wedge \Phi^{\prime} \pqty{ \frac{\sqrt{- 1}}{2 \pi} \Xi } \\
		&= \sum_{m = 0}^{n}
			\int_{S} \pqty{ - \frac{\lambda}{2 \pi} \eta } \wedge \pqty{ - \frac{\lambda}{2 \pi} d \eta }^{n - m}
			\wedge \Phi^{\prime} \pqty{ \frac{\sqrt{- 1}}{2 \pi} \Xi }.
	\end{align}
	This gives the desired conclusion.
\end{proof}

Let $(L, h)$ be a Hermitian line bundle over an $n$-dimensional complex manifold $Y$
such that $\omega = - \sqrt{- 1} \Theta_{h}$ defines a K\"{a}hler-Einstein metric on $Y$ with Einstein constant $(n + 1) \lambda$.
Consider the circle bundle $(S, T^{1, 0} S, \eta)$ associated with $(Y, L, h)$.
The Chern tensor $\tensor{S}{_{\alpha}^{\beta}_{\rho}_{\ovxs}}$ coincides with
the Bochner tensor $\tensor{B}{_{\alpha}^{\beta}_{\rho}_{\ovxs}}$ of $(Y, \omega)$.
Recall that $\Pi$ is the curvature form of the K\"{a}hler metric.
Define $\End(T^{1, 0} Y)$-valued two-forms $B$ and $K$ by
\begin{equation}
	\tensor{B}{_{\alpha}^{\beta}}
	= \tensor{B}{_{\alpha}^{\beta}_{\rho}_{\ovxs}} d z^{\rho} \wedge d \ovz^{\sigma},
	\qquad
	\tensor{K}{_{\alpha}^{\beta}}
	= \tensor{\Pi}{_{\alpha}^{\beta}} - \tensor{B}{_{\alpha}^{\beta}}
	= \lambda (\tensor{\delta}{_{\alpha}^{\beta}} \tensor{g}{_{\rho}_{\ovxs}}
		+ \tensor{g}{_{\alpha}_{\ovxs}} \tensor{\delta}{_{\rho}^{\beta}}) d z^{\rho} \wedge d \ovz^{\sigma}.
\end{equation}

\begin{lemma}
\label{lem:trace-of-Bochner-curvature}
	\begin{equation}
		\bqty{ \ch_{k} \pqty{ \frac{\sqrt{- 1}}{2 \pi} B } }
		= \sum_{j = 0}^{k} \frac{1}{(k - j) !} (\lambda c_{1}(L))^{k - j} \ch_{j}(T^{1, 0} Y \oplus \mathbb{C}).
	\end{equation}
\end{lemma}

\begin{proof}
	The two-forms $\Pi$ and $K$ satisfy the following equations:
	\begin{gather}
		\tr \Pi
		= - \sqrt{- 1} (n + 1) \lambda \omega,
		\qquad
		\tr K
		= - \sqrt{- 1} (n + 1) \lambda \omega, \\
		\Pi \wedge K
		= K \wedge \Pi
		= - \sqrt{- 1} \lambda \omega \wedge \Pi,
		\qquad
		K \wedge K
		= - \sqrt{- 1} \lambda \omega \wedge K.
	\end{gather}
	Hence
	\begin{align}
		\pqty{ \frac{\sqrt{- 1}}{2 \pi} B }^{k}
		&= \sum_{j = 1}^{k} \binom{k}{j} \pqty{- \frac{\sqrt{- 1}}{2 \pi} K}^{k - j} \wedge \pqty{\frac{\sqrt{- 1}}{2 \pi} \Pi}^{j}
			+ \pqty{- \frac{\sqrt{- 1}}{2 \pi} K}^{k} \\
		&= \sum_{j = 1}^{k} \binom{k}{j} \pqty{- \frac{\lambda}{2 \pi} \omega}^{k - j} \wedge \pqty{\frac{\sqrt{- 1}}{2 \pi} \Pi}^{j}
			+ \pqty{- \frac{\lambda}{2 \pi} \omega}^{k - 1} \wedge \pqty{- \frac{\sqrt{- 1}}{2 \pi} K}
	\end{align}
	and
	\begin{equation}
		\ch_{k} \pqty{ \frac{\sqrt{- 1}}{2 \pi} B }
		= \sum_{j = 1}^{l} \frac{1}{(k - j) !} \pqty{- \frac{\lambda}{2 \pi} \omega}^{k - j}
			\wedge \frac{1}{j !}\tr \pqty{\frac{\sqrt{- 1}}{2 \pi} \Pi}^{j}
			+ \frac{n + 1}{k !} \pqty{- \frac{\lambda}{2 \pi} \omega}^{k}.
	\end{equation}
	Therefore we have
	\begin{align}
		\bqty{ \ch_{k} \pqty{ \frac{\sqrt{- 1}}{2 \pi} B } }
		&= \sum_{j = 1}^{k} \frac{1}{(k - j) !} (\lambda c_{1}(L))^{k - j}
			\ch_{j}(T^{1, 0} Y)
			+ \frac{n + 1}{k !} (\lambda c_{1}(L))^{k} \\
		&= \sum_{j = 0}^{k} \frac{1}{(k - j) !} (\lambda c_{1}(L))^{k - j} \ch_{j}(T^{1, 0} Y \oplus \mathbb{C}),
	\end{align}
	which completes the proof.
\end{proof}

\begin{remark}
	Formally,
	\begin{equation}
		\bqty{ \ch_{k} \pqty{ \frac{\sqrt{- 1}}{2 \pi} B } }
		= \ch_{k}((T^{1, 0} Y \oplus \mathbb{C}) \otimes L^{\lambda}).
	\end{equation}
\end{remark}

\begin{proof}[Proof of \cref{thm:Marugame's-CR-invariants-of-Chern-character-type}]
	In our setting,
	$d \eta = p^{\ast} \omega$ and $\Xi = p^{\ast} B$.
	From \cref{thm:Marugame's-CR-invariant-for-Sasakian-eta-Einstein-manifolds},
	we obtain
	\begin{align}
		\scrI_{\varsigma}(S)
		&= \int_{S} \pqty{ - \frac{\lambda}{2 \pi} \eta } \wedge \pqty{ - \frac{\lambda}{2 \pi} d \eta }^{\varsigma_{1}}
			\wedge \Phi_{\varsigma} \pqty{ \frac{\sqrt{- 1}}{2 \pi} \Xi } \\
		&= \int_{S} \pqty{ - \frac{\lambda}{2 \pi} \eta }
			\wedge p^{\ast} \bqty{ \pqty{ - \frac{\lambda}{2 \pi} \omega }^{\varsigma_{1}}
			\wedge \Phi_{\varsigma} \pqty{ \frac{\sqrt{- 1}}{2 \pi} B} } \\
		&= - \lambda \int_{Y} \pqty{ - \frac{\lambda}{2 \pi} \omega }^{\varsigma_{1}}
			\wedge \Phi_{\varsigma} \pqty{ \frac{\sqrt{- 1}}{2 \pi} B } \\
		&= - \lambda \int_{Y} (\lambda c_{1}(L))^{\varsigma_{1}}
			\prod_{k = 2}^{n} \bqty{\sum_{j = 0}^{k} \frac{1}{(k - j) !} (\lambda c_{1}(L))^{k - j}
			\ch_{j}(T^{1, 0} Y \oplus \mathbb{C})}^{\varsigma_{k}};
	\end{align}
	here,
	the third equality follows from integration along fibers.
\end{proof}

As an application,
we show the algebraic independence of $(\scrI_{\varsigma})_{\varsigma \in \Part(n)}$.
To this end,
we consider a smooth complete intersection variety $Y_{d}$ of multi-degree $d = (d_{1}, \dots , d_{n})$ in $\cps^{2 n}$.
Take as $L$ the restriction of $\mathcal{O}(- 1)$ to $Y_{d}$,
and denote by $\tau$ the first Chern class of $L^{- 1}$ for simplicity.
The normal bundle of $Y_{d}$ is isomorphic to $\bigoplus_{i = 1}^{n} \mathcal{O}(d_{i}) |_{Y_{d}}$,
and so
\begin{equation}
\label{eq:chern-character-of-complete-intersection}
	\begin{split}
		\ch_{k}(T^{1, 0} Y_{d} \oplus \mathbb{C})
		&= - \ch_{k} \pqty{ \bigoplus_{i = 1}^{n} \mathcal{O}(d_{i}) |_{Y_{d}} } + \ch_{k}(T^{1, 0} \cps^{2 n} |_{Y_{d}} \oplus \mathbb{C}) \\
		&= - \ch_{k} \pqty{ \bigoplus_{i = 1}^{n} \mathcal{O}(d_{i}) |_{Y_{d}} } + \ch_{k}(\mathcal{O}(1)^{\oplus (2 n + 1)} |_{Y_{d}}) \\
		&= - \sum_{i = 1}^{n} \ch_{k}(\mathcal{O}(d_{i}) |_{Y_{d}}) + (2 n + 1) \ch_{k}(\mathcal{O}(1) |_{Y_{d}}) \\
		&= (- s_{k}(d) + C_{k}) \tau^{k},
	\end{split}
\end{equation}
where
\begin{equation}
	s_{k}(d)
	= \frac{1}{k !} (d_{1}^{k} + \dots + d_{n}^{k}),
	\qquad
	C_{k}
	=
	\frac{2 n + 1}{k !}.
\end{equation}
In particular,
\begin{equation}
	c_{1}(T^{1, 0} Y_{d})
	= \ch_{1}(T^{1, 0} Y_{d} \oplus \mathbb{C})
	= (- s_{1}(d) + 2 n + 1) \tau.
\end{equation}
If $s_{1}(d) > 2 n + 1$,
then there exists a Hermitian metric $h$ on $L$
such that $\omega = - \sqrt{- 1} \Theta_{h}$ defines a K\"{a}hler-Einstein manifold
with Einstein constant $(n + 1) \lambda_{d} = - s_{1}(d) + 2 n + 1$
by the Aubin-Yau theorem.
Set
\begin{equation}
	\nu_{k}(d)
	= - \sum_{j = 1}^{k} \frac{s_{1}(d)^{k - j} s_{j}(d)}{(n + 1)^{k - j} (k - j) !} + \frac{s_{1}(d)^{k}}{(n + 1)^{k - 1} k !},
\end{equation}
which is a homogeneous symmetric polynomial in $d$ of degree $k$.
The cohomology class
\begin{equation}
	\bqty{ \ch_{k} \pqty{ \frac{\sqrt{- 1}}{2 \pi} B } }
	= \sum_{j = 0}^{k} \frac{1}{(k - j) !} (\lambda_{d} c_{1}(L))^{k - j} \ch_{j}(T^{1, 0} Y_{d} \oplus \mathbb{C})
\end{equation}
is equal to
\begin{equation}
	[\nu_{k}(d) + (\text{a symmetric polynomial in $d$ of degree at most $k - 1$})] \tau^{k};
\end{equation}
this follows from \cref{eq:chern-character-of-complete-intersection}.
For $\varsigma \in \Part(n)$,
define a homogeneous symmetric polynomial $p_{\varsigma}(d)$ in $d$ of degree $n$ by
\begin{equation}
	p_{\varsigma}(d)
	= \pqty{\frac{s_{1}(d)}{n + 1}}^{\varsigma_{1}} \prod_{k = 2}^{n} \nu_{k}(d)^{\varsigma_{k}}.
\end{equation}
Similar to the above,
the cohomology class
\begin{equation}
	(\lambda_{d} c_{1}(L))^{\varsigma_{1}}
	\prod_{k = 2}^{n} \bqty{\sum_{j = 0}^{k} \frac{1}{(k - j) !} (\lambda_{d} c_{1}(L))^{k - j}
	\ch_{j}(T^{1, 0} Y_{d} \oplus \mathbb{C})}^{\varsigma_{k}}
\end{equation}
is of the form
\begin{equation}
	[p_{\varsigma}(d) + (\text{a symmetric polynomial in $d$ of degree at most $n - 1$})] \tau^{n}.
\end{equation}
Denote by $S_{d}$ the circle bundle associated with $(Y_{d}, L, h)$.
From \cref{thm:Marugame's-CR-invariants-of-Chern-character-type} and $\int_{Y_{d}} \tau^{n} = d_{1} \dotsm d_{n}$,
it follows that the invariant $\scrI_{\varsigma}(S_{d})$ is a symmetric polynomial in $d$,
and its leading term is given by
\begin{equation}
	\frac{d_{1} \dotsm d_{n} s_{1}(d)}{n + 1} p_{\varsigma}(d).
\end{equation}

\begin{proof}[Proof of \cref{thm:algebraically-independence}]
	Suppose that there exists a non-constant polynomial $f(x_{\varsigma})$
	such that $f(\scrI_{\varsigma})$ is identically zero.
	Let $m (\geq 1)$ be the degree of $f$,
	and $f_{m}$ the leading part of $f$,
	which is a non-trivial homogeneous polynomial of degree $m$.
	From $f(\scrI_{\varsigma}(S_{d})) = 0$,
	it follows that
	\begin{equation}
		\pqty{\frac{d_{1} \dotsm d_{n} s_{1}(d)}{n + 1}}^{m} f_{m} (p_{\varsigma}(d))
		= 0,
	\end{equation}
	or equivalently,
	$f_{m} (p_{\varsigma}(d)) = 0$ as a polynomial in $d$.
	The polynomial $p_{\varsigma}$
	is written as a linear combination of
	$(s_{1}^{\varsigma^{\prime}_{1}} \dotsm s_{n}^{\varsigma^{\prime}_{n}})_{\varsigma^{\prime} \in \Part(n)}$.
	Conversely,
	for any $\varsigma^{\prime} \in \Part(n)$,
	the polynomial $s_{1}^{\varsigma^{\prime}_{1}} \dots s_{n}^{\varsigma^{\prime}_{n}}$
	is a linear combination of $(p_{\varsigma})_{\varsigma \in \Part(n)}$.
	Thus we obtain a non-trivial homogeneous polynomial $g_{m}(x_{\varsigma})$ of degree $m$
	such that $g_{m}(s_{1}^{\varsigma_{1}} \dotsm s_{n}^{\varsigma_{n}}) = 0$ as a polynomial in $d$.
	This contradicts the fact that
	$s_{1}, \dots , s_{n}$ are algebraically independent over $\mathbb{C}$.
	Therefore the invariants $(\scrI_{\varsigma})_{\varsigma \in \Part(n)}$ do not satisfy
	any non-trivial polynomial relation with coefficient in $\mathbb{C}$.
\end{proof}

\cref{thm:algebraically-independence} gives a criterion for the triviality of $\scrI_{\Phi}$.

\begin{proof}[Proof of \cref{cor:triviality-of-Marugame's-invariant}]
	We have already shown that $\scrI_{\Phi}$ is trivial if $\Phi \equiv 0$ modulo $\ch_{1}$ in \cref{lem:linear-combination-of-Marugame's-invariant}.
	Conversely,
	assume that $\scrI_{\Phi}$ is trivial.
	It follows from \cref{lem:linear-combination-of-Marugame's-invariant} that
	\begin{equation}
		\sum_{\varsigma \in \Part(n)} C^{\Phi}_{\varsigma} \scrI_{\varsigma}
		= 0.
	\end{equation}
	\cref{thm:algebraically-independence} implies that
	the coefficients $C^{\Phi}_{\varsigma}$ are zero,
	and so $\Phi = \ch_{1} \wtxcph$.
\end{proof}

\section{Burns-Epstein invariant for Sasakian $\eta$-Einstein manifolds}
\label{section:Burns-Epstein-invariant-for-Sasakian-eta-Einstein-manifolds}

Let $\Omega$ be an $(n + 1)$-dimensional strictly pseudoconvex domain with boundary $M$,
which admits a pseudo-Einstein contact form.
Fix a Fefferman defining function $\rho$ of $\Omega$
and consider the renormalized connection.
The \emph{Burns-Epstein invariant} $\mu(M)$ is defined by
\begin{equation}
	\mu(M)
	= \frac{1}{n !} \pqty{\frac{\sqrt{- 1}}{2 \pi}}^{n + 1} \sum_{k = 0}^{n} \binom{n}{k}
		\int_{M} (\Phi_{k}^{(0)} - \Phi_{k}^{(1)}),
\end{equation}
where $\mathfrak{S}_{n}$ denotes the symmetric group of order $n$ and
\begin{align}
	\Phi_{0}^{(0)}
	&= \sum_{\sigma, \tau \in \mathfrak{S}_{n}} \sgn(\sigma \tau)
		\tensor{\theta}{_{\infty}^{\infty}} \wedge \tensor{\theta}{_{\sigma(1)}^{\infty}} \wedge \tensor{\theta}{_{\infty}^{\tau(1)}} \wedge
		\dotsm \wedge \tensor{\theta}{_{\sigma(n)}^{\infty}} \wedge \tensor{\theta}{_{\infty}^{\tau(n)}}, \\
	\Phi_{k}^{(0)}
	&= \sum_{\sigma, \tau \in \mathfrak{S}_{n}} \sgn(\sigma \tau)
		\tensor{\theta}{_{\infty}^{\infty}} \wedge \tensor{\Theta}{_{\sigma(1)}^{\tau(1)}} \wedge \dotsm \wedge \tensor{\Theta}{_{\sigma(k)}^{\tau(k)}} \\
	& \qquad \qquad \wedge \tensor{\theta}{_{\sigma(k + 1)}^{\infty}} \wedge \tensor{\theta}{_{\infty}^{\tau(k + 1)}} \wedge 
		\dotsm \wedge \tensor{\theta}{_{\sigma(n)}^{\infty}} \wedge \tensor{\theta}{_{\infty}^{\tau(n)}}
		\quad (1 \leq k \leq n), \\
	\Phi_{0}^{(1)}
	&= \sum_{\sigma, \tau \in \mathfrak{S}_{n}} \sgn(\sigma \tau)
		\tensor{\Theta}{_{\sigma(1)}^{\infty}} \wedge \tensor{\theta}{_{\infty}^{\tau(1)}} \wedge
		\tensor{\theta}{_{\sigma(2)}^{\infty}} \wedge \tensor{\theta}{_{\infty}^{\tau(2)}} \wedge
		\dotsm \wedge \tensor{\theta}{_{\sigma(n)}^{\infty}} \wedge \tensor{\theta}{_{\infty}^{\tau(n)}}, \\
	\Phi_{k}^{(1)}
	&= \sum_{\sigma, \tau \in \mathfrak{S}_{n}} \sgn(\sigma \tau)
		 \tensor{\Theta}{_{\sigma(1)}^{\infty}} \wedge \tensor{\theta}{_{\infty}^{\tau(1)}}\wedge
		 \tensor{\Theta}{_{\sigma(2)}^{\tau(2)}} \wedge
		 \dotsm \wedge \tensor{\Theta}{_{\sigma(k + 1)}^{\tau(k + 1)}} \\
	& \qquad \qquad \wedge \tensor{\theta}{_{\sigma(k + 2)}^{\infty}} \wedge \tensor{\theta}{_{\infty}^{\tau(k + 2)}}
		\dotsm \wedge \tensor{\theta}{_{\sigma(n)}^{\infty}} \wedge \tensor{\theta}{_{\infty}^{\tau(n)}}
		\quad (1 \leq k \leq n - 1), \\
	\Phi_{n}^{(1)}
	&= 0.
\end{align}
This $\mu(M)$ is independent of the choice of a Fefferman defining function,
and gives a global CR invariant of $M$~\cite{Marugame2016}*{Theorem 4.6}.

In the previous section,
we have given formulae of the renormalized connection and curvature
for Sasakian $\eta$-Einstein manifolds.
Combining those with some algebraic facts yields a proof of \cref{thm:Burns-Epstein-invariant-for-Sasakian-eta-Einstein-manifolds}.

\begin{proof}[Proof of \cref{thm:Burns-Epstein-invariant-for-Sasakian-eta-Einstein-manifolds}]
	By \cref{lem:renormalized-connection-for-Sasakian-eta-Einstein},
	the renormalized connection and curvature on $S$ are given by
	\begin{gather}
		\tensor{\theta}{_{\alpha}^{\beta}}
		= \tensor{\omega}{_{\alpha}^{\beta}} + \sqrt{- 1} \lambda \eta \tensor{\delta}{_{\alpha}^{\beta}},
		\qquad
		\tensor{\theta}{_{\infty}^{\beta}}
		= \lambda \tensor{\theta}{^{\beta}}, \\
		\tensor{\theta}{_{\alpha}^{\infty}}
		= - \tensor{l}{_{\alpha}_{\ovxg}} \tensor{\theta}{^{\ovxg}},
		\qquad
		\tensor{\theta}{_{\infty}^{\infty}}
		= \sqrt{- 1} \lambda \eta, \\
		\tensor{\Theta}{_{\alpha}^{\beta}}
		= \tensor{\Omega}{_{\alpha}^{\beta}}
		+ \sqrt{- 1} \lambda d \eta \cdot \tensor{\delta}{_{\alpha}^{\beta}}
		- \lambda \tensor{l}{_{\alpha}_{\overline{\gamma}}} \tensor{\theta}{^{\beta}} \wedge \tensor{\theta}{^{\overline{\gamma}}}, \\
		\tensor{\Theta}{_{\infty}^{\beta}}
		= 0,
		\qquad
		\tensor{\Theta}{_{\alpha}^{\infty}}
		= 0,
		\qquad
		\tensor{\Theta}{_{\infty}^{\infty}}
		= 0.
	\end{gather}
	Since $\tensor{\Theta}{_{\alpha}^{\infty}} = 0$,
	the $(2 n + 1)$-form $\Phi_{k}^{(1)}$ vanishes identically for any $0 \leq k \leq n$.
	On the other hand,
	\begin{align}
		\frac{1}{n !} \pqty{\frac{\sqrt{- 1}}{2 \pi}}^{n + 1} \sum_{k = 0}^{n} \binom{n}{k} \Phi_{k}^{(0)}
		&= \pqty{\frac{\sqrt{- 1}}{2 \pi}}^{n + 1}
			\tensor{\theta}{_{\infty}^{\infty}} \wedge
			\det \pqty{ \tensor{\Theta}{_{\alpha}^{\beta}} + \tensor{\theta}{_{\alpha}^{\infty}} \wedge \tensor{\theta}{_{\infty}^{\beta}} } \\
		&= \pqty{ - \frac{\lambda}{2 \pi} \eta } \wedge
			\det \pqty{ - \frac{\lambda}{2 \pi} d \eta \cdot \tensor{\delta}{_{\alpha}^{\beta}}
			+ \frac{\sqrt{-1}}{2 \pi} \tensor{\Omega}{_{\alpha}^{\beta}} } \\
		&= \sum_{m = 0}^{n} \pqty{ - \frac{\lambda}{2 \pi} \eta }
			\wedge \pqty{ - \frac{\lambda}{2 \pi} d \eta }^{n - m}
			\wedge c_{m} \pqty{ \frac{\sqrt{- 1}}{2 \pi} \Omega }.
	\end{align}
	Hence
	\begin{equation}
		\mu(S)
		= \sum_{m = 0}^{n} \int_{S} \pqty{ - \frac{\lambda}{2 \pi} \eta }
			\wedge \pqty{ - \frac{\lambda}{2 \pi} d \eta }^{n - m}
			\wedge c_{m} \pqty{ \frac{\sqrt{- 1}}{2 \pi} \Omega },
	\end{equation}
	which completes the proof.
\end{proof}

Similar to the proof of \cref{thm:Marugame's-CR-invariants-of-Chern-character-type},
we obtain a formula of the Burns-Epstein invariant
for the circle bundle case in terms of the Einstein constant and characteristic numbers.

\begin{proof}[Proof of \cref{cor:Burns-Epstein-invariant-for-tubes}]
	In this setting,
	integration along fibers gives that
	\begin{align}
		\mu(S)
		&= - \lambda \sum_{m = 0}^{n}
			\int_{Y} \pqty{ - \frac{\lambda}{2 \pi} \omega }^{n - m} \wedge c_{m} \pqty{ \frac{\sqrt{- 1}}{2 \pi} \Pi } \\
		&= - \lambda \sum_{m = 0}^{n}
			\int_{Y} (\lambda c_{1}(L))^{n - m} c_{m}(T^{1, 0} Y).
	\end{align}
	This gives the desired formula.
\end{proof}

\section{Concluding remarks}
\label{section:concluding-remarks}

In this section,
we propose a conjecture on the Burns-Epstein invariant and global CR invariants via renormalized characteristic forms.

In \cref{thm:Marugame's-CR-invariant-for-Sasakian-eta-Einstein-manifolds,thm:Burns-Epstein-invariant-for-Sasakian-eta-Einstein-manifolds},
we computed global CR invariants via renormalized characteristic forms and the Burns-Epstein invariant
for Sasakian $\eta$-Einstein manifolds.
Similar to the proof in \cref{lem:trace-of-Bochner-curvature},
we have the following equality for any closed Sasakian $\eta$-Einstein manifold $(S, T^{1, 0} S, \eta)$ of dimension $2 n + 1$:
\begin{equation}
\label{eq:relation-between-BE-and-Marugame-for-Sasaki-Einstein}
	\mu(S)
	= \sum_{\varsigma \in \Part(n)} C_{\varsigma} \scrI_{\varsigma}(S),
\end{equation}
where $C_{\varsigma}$ is a real constant depending only on $\varsigma$.
In general dimensons,
it is hard to compute the coefficient $C_{\varsigma}$ explicitly.
However,
we can do it in low dimensions.
When $n \leq 3$,
we have
\begin{equation}
	\mu(S)
	=
	\begin{cases}
		- \scrI_{(1)}(S) & n = 1, \\
		\scrI_{(2, 0)}(S) - \scrI_{(0, 1)}(S) & n = 2, \\
		- \scrI_{(3, 0, 0)}(S) + \scrI_{(1, 1, 0)}(S) + 2 \scrI_{(0, 0, 1)}(S) & n = 3.
	\end{cases}
\end{equation}
When $n = 4$,
a computation shows that
\begin{equation}
	\mu(S)
	= \scrI_{(4, 0, 0, 0)}(S) - \scrI_{(2, 1, 0, 0)}(S) - 2 \scrI_{(1, 0, 1, 0)}(S)
		+ \frac{1}{2} \scrI_{(0, 2, 0, 0)}(S) - 6 \scrI_{(0, 0, 0, 1)}(S).
\end{equation}
It is natural to expect that the same equality as above holds for general closed strictly pseudoconvex CR manifolds.

\begin{conjecture}
\label{conj:relation-between-BE-and-Marugame}
	Fix a positive integer $n$.
	There exists a family $(C_{\varsigma})_{\varsigma \in \Part(n)}$ of real numbers such that
	\begin{equation}
		\mu(M)
		= \sum_{\varsigma \in \Part(n)} C_{\varsigma} \scrI_{\varsigma}(M)
	\end{equation}
	for any closed strictly pseudoconvex CR manifold $(M, T^{1, 0} M)$ of dimension $2 n + 1$
	admitting a pseudo-Einstein contact form.
\end{conjecture}

Here we show that this conjecture holds in low dimensions.

\begin{proposition}
	\cref{conj:relation-between-BE-and-Marugame} is true for $n = 1$ and $2$.
\end{proposition}

\begin{proof}
	We first note that
	\begin{equation}
		\scrI_{(n, 0, \dots, 0)}(M)
		= \frac{(- 1)^{n + 1}}{2 (2 \pi)^{n + 1} (n !)^{2}} \ovQ^{\prime}(M),
	\end{equation}
	where $\ovQ^{\prime}(M)$ is the total $Q^{\prime}$-curvature;
	see~\cite{Hirachi2014}*{Theorem 5.6} for a proof.
	In dimension three,
	$\ovQ^{\prime}(M) = - 8 \pi^{2} \mu(M)$~\cite{Hirachi2014}*{Theorem 6.6},
	and so
	\begin{equation}
		\mu(M)
		= - \scrI_{(1)}(M).
	\end{equation}
	In dimension five,
	$\scrI_{(0, 1)}(M)$ is equal to
	\begin{equation}
		\scrI_{(0, 1)}(M)
		= \frac{1}{16 \pi^{3}} \overline{\calI}^{\prime}(M),
	\end{equation}
	where $\overline{\calI}^{\prime}(M)$ is the total $\calI^{\prime}$-curvature
	introduced by Case and Gover~\cite{Case-Gover2020}*{Proposition 8.8};
	this follows from~\cite{Marugame2021}*{Section 5.6 and Theorem 6.6}.
	On the other hand,
	it is known that
	\begin{equation}
		\ovQ^{\prime}(M) + 64 \pi^{3} \mu(M)
		= - 4 \overline{\calI}^{\prime}(M);
	\end{equation}
	see~\cite{Hirachi-Marugame-Matsumoto2017}*{(1.11)}.
	Thus we have
	\begin{equation}
		\mu(M)
		= \scrI_{(2, 0)}(M) - \scrI_{(0, 1)}(M),
	\end{equation}
	which completes the proof.
\end{proof}

\section*{Acknowledgements}

The author would like to thank Jeffrey S. Case and Taiji Marugame for helpful comments.

\bibliography{my-reference,my-reference-preprint}

\begin{bibdiv}
\begin{biblist}

\bib{Boutet_de_Monvel1975}{incollection}{
      author={Boutet~de Monvel, L.},
       title={Int\'egration des \'equations de {C}auchy-{R}iemann induites
  formelles},
        date={1975},
   booktitle={S\'eminaire {G}oulaouic-{L}ions-{S}chwartz 1974--1975;
  \'equations aux deriv\'ees partielles lin\'eaires et non lin\'eaires},
   publisher={Centre Math., \'Ecole Polytech., Paris},
       pages={Exp. No. 9, 14},
      review={\MR{0409893}},
}

\bib{Burns-Epstein1990-Char}{article}{
      author={Burns, D.},
      author={Epstein, C.~L.},
       title={Characteristic numbers of bounded domains},
        date={1990},
        ISSN={0001-5962},
     journal={Acta Math.},
      volume={164},
      number={1-2},
       pages={29\ndash 71},
         url={https://doi.org/10.1007/BF02392751},
      review={\MR{1037597}},
}

\bib{Boyer-Galicki2008}{book}{
      author={Boyer, Charles~P.},
      author={Galicki, Krzysztof},
       title={Sasakian geometry},
      series={Oxford Mathematical Monographs},
   publisher={Oxford University Press, Oxford},
        date={2008},
        ISBN={978-0-19-856495-9},
      review={\MR{2382957}},
}

\bib{Case-Gover2020}{article}{
      author={Case, Jeffrey~S.},
      author={Gover, A.~Rod},
       title={The {$P'$}-operator, the {$Q'$}-curvature, and the {CR} tractor
  calculus},
        date={2020},
        ISSN={0391-173X},
     journal={Ann. Sc. Norm. Super. Pisa Cl. Sci. (5)},
      volume={20},
      number={2},
       pages={565\ndash 618},
      review={\MR{4105911}},
}

\bib{Case-Takeuchi2023}{article}{
      author={Case, Jeffrey~S.},
      author={Takeuchi, Yuya},
       title={{$\mathcal{I}^{\prime}$}-curvatures in higher dimensions and the
  {H}irachi conjecture},
        date={2023},
        ISSN={0025-5645},
     journal={J. Math. Soc. Japan},
      volume={75},
      number={1},
       pages={291\ndash 328},
         url={https://doi.org/10.2969/jmsj/87718771},
      review={\MR{4539017}},
}

\bib{Case-Yang2013}{article}{
      author={Case, Jeffrey~S.},
      author={Yang, Paul},
       title={A {P}aneitz-type operator for {CR} pluriharmonic functions},
        date={2013},
        ISSN={2304-7909},
     journal={Bull. Inst. Math. Acad. Sin. (N.S.)},
      volume={8},
      number={3},
       pages={285\ndash 322},
      review={\MR{3135070}},
}

\bib{Fefferman1974}{article}{
      author={Fefferman, Charles},
       title={The {B}ergman kernel and biholomorphic mappings of pseudoconvex
  domains},
        date={1974},
        ISSN={0020-9910},
     journal={Invent. Math.},
      volume={26},
       pages={1\ndash 65},
         url={https://doi.org/10.1007/BF01406845},
      review={\MR{0350069}},
}

\bib{Graham-Lee1988}{article}{
      author={Graham, C.~Robin},
      author={Lee, John~M.},
       title={Smooth solutions of degenerate {L}aplacians on strictly
  pseudoconvex domains},
        date={1988},
        ISSN={0012-7094},
     journal={Duke Math. J.},
      volume={57},
      number={3},
       pages={697\ndash 720},
         url={http://dx.doi.org/10.1215/S0012-7094-88-05731-6},
      review={\MR{975118}},
}

\bib{Hirachi2014}{article}{
      author={Hirachi, Kengo},
       title={{$Q$}-prime curvature on {CR} manifolds},
        date={2014},
        ISSN={0926-2245},
     journal={Differential Geom. Appl.},
      volume={33},
      number={suppl.},
       pages={213\ndash 245},
         url={http://dx.doi.org/10.1016/j.difgeo.2013.10.013},
      review={\MR{3159959}},
}

\bib{Hirachi1993}{incollection}{
      author={Hirachi, Kengo},
       title={Scalar pseudo-{H}ermitian invariants and the {S}zeg{\H{o}} kernel
  on three-dimensional {CR} manifolds},
        date={1993},
   booktitle={Complex geometry ({O}saka, 1990)},
      series={Lecture Notes in Pure and Appl. Math.},
      volume={143},
   publisher={Dekker, New York},
       pages={67\ndash 76},
      review={\MR{1201602}},
}

\bib{Harvey-Lawson1975}{article}{
      author={Harvey, F.~Reese},
      author={Lawson, H.~Blaine, Jr.},
       title={On boundaries of complex analytic varieties. {I}},
        date={1975},
        ISSN={0003-486X},
     journal={Ann. of Math. (2)},
      volume={102},
      number={2},
       pages={223\ndash 290},
         url={https://doi.org/10.2307/1971032},
      review={\MR{0425173}},
}

\bib{Hirachi-Marugame-Matsumoto2017}{article}{
      author={Hirachi, Kengo},
      author={Marugame, Taiji},
      author={Matsumoto, Yoshihiko},
       title={Variation of total {Q}-prime curvature on {CR} manifolds},
        date={2017},
        ISSN={0001-8708},
     journal={Adv. Math.},
      volume={306},
       pages={1333\ndash 1376},
         url={http://dx.doi.org/10.1016/j.aim.2016.11.005},
      review={\MR{3581332}},
}

\bib{Hislop-Perry-Tang2008}{article}{
      author={Hislop, Peter~D.},
      author={Perry, Peter~A.},
      author={Tang, Siu-Hung},
       title={C{R}-invariants and the scattering operator for complex manifolds
  with boundary},
        date={2008},
        ISSN={1948-206X},
     journal={Anal. PDE},
      volume={1},
      number={2},
       pages={197\ndash 227},
         url={http://dx.doi.org/10.2140/apde.2008.1.197},
      review={\MR{2472889}},
}

\bib{Lee1988}{article}{
      author={Lee, John~M.},
       title={Pseudo-{E}instein structures on {CR} manifolds},
        date={1988},
        ISSN={0002-9327},
     journal={Amer. J. Math.},
      volume={110},
      number={1},
       pages={157\ndash 178},
         url={http://dx.doi.org/10.2307/2374543},
      review={\MR{926742}},
}

\bib{Lempert1995}{article}{
      author={Lempert, L\'aszl\'o},
       title={Algebraic approximations in analytic geometry},
        date={1995},
        ISSN={0020-9910},
     journal={Invent. Math.},
      volume={121},
      number={2},
       pages={335\ndash 353},
         url={https://doi.org/10.1007/BF01884302},
      review={\MR{1346210}},
}

\bib{Marugame2016}{article}{
      author={Marugame, Taiji},
       title={Renormalized {C}hern-{G}auss-{B}onnet formula for complete
  {K}\"ahler-{E}instein metrics},
        date={2016},
        ISSN={0002-9327},
     journal={Amer. J. Math.},
      volume={138},
      number={4},
       pages={1067\ndash 1094},
         url={https://doi.org/10.1353/ajm.2016.0034},
      review={\MR{3538151}},
}

\bib{Marugame2021}{article}{
      author={Marugame, Taiji},
       title={Renormalized characteristic forms of the {C}heng-{Y}au metric and
  global {CR} invariants},
        date={2021},
        ISSN={0001-8708},
     journal={Adv. Math.},
      volume={377},
       pages={107468},
         url={https://doi.org/10.1016/j.aim.2020.107468},
      review={\MR{4186011}},
}

\bib{Takeuchi2018}{article}{
      author={Takeuchi, Yuya},
       title={Ambient constructions for {S}asakian {$\eta$}-{E}instein
  manifolds},
        date={2018},
        ISSN={0001-8708},
     journal={Adv. Math.},
      volume={328},
       pages={82\ndash 111},
         url={https://doi.org/10.1016/j.aim.2018.01.007},
      review={\MR{3771124}},
}

\bib{Cap-Gover2019}{article}{
      author={\v{C}ap, Andreas},
      author={Gover, A.~Rod},
       title={c-projective compactification; (quasi-){K}\"{a}hler metrics and
  {CR} boundaries},
        date={2019},
        ISSN={0002-9327},
     journal={Amer. J. Math.},
      volume={141},
      number={3},
       pages={813\ndash 856},
         url={https://doi.org/10.1353/ajm.2019.0017},
      review={\MR{3956522}},
}

\end{biblist}
\end{bibdiv}

\end{document}